\documentclass{amsart}

\usepackage{amsmath, amsthm}
\usepackage{amssymb}
\usepackage{amsfonts}
\usepackage{latexsym}
\usepackage{mathrsfs}
\usepackage{bm}
\usepackage{graphicx}
\usepackage{color}
\usepackage{hyperref}

\theoremstyle{plain} 
\newtheorem{theorem}{Theorem}[section]
\newtheorem{lemma}[theorem]{Lemma}
\newtheorem{proposition}[theorem]{Proposition}
\newtheorem{excer}[theorem]{Excercises}
\newtheorem{remark}[theorem]{Remark}

\newtheorem{corollary}[theorem]{Corollary}
\newtheorem*{rem*}{Remark}
\newtheorem*{def*}{Definition}
\newtheorem*{thm*}{Theorem}

\newcommand{\be}{\begin{equation}}
\newcommand{\ene}{\end{equation}}
\newcommand{\br}{\begin{remark}}
\newcommand{\er}{\end{remark}}
\newcommand{\bl}{\begin{lem}}
\newcommand{\el}{\end{lem}}
\newcommand{\bcor}{\begin{cor}}
\newcommand{\ecor}{\end{cor}}
\newcommand{\bpro}{\begin{pro}}
\newcommand{\epro}{\end{pro}}
\newcommand{\ben}{\begin{enumerate}}
\newcommand{\een}{\end{enumerate}}
\newcommand{\bp}{\begin{proof}}
\newcommand{\ep}{\end{proof}}
\newcommand{\bpo}{\begin{pro}}
\newcommand{\epo}{\end{pro}}
\newcommand{\beq}{\begin{equation*}}
\newcommand{\eeq}{\end{equation*}}
\newcommand{\bear}{\begin{eqnarray}}
\newcommand{\eear}{\end{eqnarray}}
\newcommand{\beqar}{\begin{eqnarray*}}
\newcommand{\eeqar}{\end{eqnarray*}}
\newcommand{\bt}{\begin{theorem}}
\newcommand{\et}{\end{theorem}}
\newcommand{\bex}{\begin{excer}}
\newcommand{\eex}{\end{excer}}

\newcommand{\R}{\mathbb{R}}

\newcommand{\sM}{\mathcal{M}}
\newcommand{\psM}{\partial \mathcal{M}}

\DeclareMathOperator{\area}{Area}
\DeclareMathOperator{\diam}{Diam}
\DeclareMathOperator{\spg}{SpG}

\begin{document}
\title[Spectral gaps for large genus]{Degenerating hyperbolic surfaces and spectral gaps for large genus}

\author{Yunhui Wu}
\address[Y. ~W. ]{Tsinghua University, Haidian District, Beijing 100084, China}
\email{yunhui\_wu@mail.tsinghua.edu.cn}

\author{Haohao Zhang}
\address[H. ~Z. ]{Tsinghua University, Haidian District, Beijing 100084, China}
\email{zhh21@mails.tsinghua.edu.cn}

\author{Xuwen Zhu}
\address[X. ~Z. ]{Northeastern University, Boston, MA 02115, USA}
\email{x.zhu@northeastern.edu}


\begin{abstract}
In this article we study the differences of two consecutive eigenvalues $\lambda_{i}-\lambda_{i-1}$ up to $i=2g-2$ for the Laplacian on hyperbolic surfaces of genus $g$, and show that the supremum of such spectral gaps over the moduli space has infimum limit at least $\frac{1}{4}$ as genus goes to infinity. A min-max principle for eigenvalues on degenerating hyperbolic surfaces is also established. 
\end{abstract}

\maketitle


\section{Introduction}\label{int}
For a closed Riemann surface $X_g$ of genus $g\geq 2$, consider the hyperbolic metric uniquely determined by its complex structure. We study 
the spectrum of the Laplacian on $X_g$, which is a discrete subset in $\mathbb{R}^{\geq 0}$ and consists of eigenvalues with finite multiplicities. The eigenvalues, counted with multiplicities, are listed in the following increasing order
\[0=\lambda_0(X_g)<\lambda_1(X_g)\leq \lambda_2(X_g) \leq \cdots \to \infty.\]
Let $\sM_g$ be the moduli space of Riemann surfaces of genus $g$, which is an open orbifold of dimension equal to $6g-6$. For each index $i$, the $i$-th eigenvalue $\lambda_i(\cdot)$ is a bounded continuous function on $\sM_g$. In this paper we study the differences of two consecutive eigenvalues and will focus on the behavior of such spectral gaps when genus $g\rightarrow \infty$.

\begin{def*}
For all $i\geq 1$, the \emph{$i$-th spectral gap} $\spg_i(\cdot)$ is a bounded continuous function over the moduli space $\sM_g$ defined as follows.
\beqar
\spg_i:\sM_g &\to& \R^{\geq 0} \\
X_g &\mapsto& \lambda_i(X_g)-\lambda_{i-1}(X_g).
\eeqar
\end{def*}
\noindent By definition $\spg_1(X_g)=\lambda_1(X_g)$. For all $i\geq 1$, the $i$-th spectral gap $\spg_i(\cdot)$ can be arbitrarily close to $0$ (e.g. see Proposition \ref{spg=0}). In this paper we mainly study the quantity $\sup \limits_{X_g \in \sM_g} \spg_{i}(X_g)$ for large $g$ and a family of indices $i$'s. 

The main result of this article is the limiting behavior of the lower bound of the spectral gaps:

\bt\label{mt-1}
Let $\{\eta(g)\}_{g=2}^{\infty}$ be any sequence of integers with $\eta(g)\in [1,2g-2]$, then
\[\liminf_{g\to \infty}  \sup_{X_g \in \sM_g} \spg_{\eta(g)}(X_g)\geq \frac{1}{4}.\]
\et
\begin{rem*} 
The sequence $\{\eta(g)\}$ is arbitrary as long as it satisfies the bounds, examples including $\eta(g)\equiv 2$, $\eta(g) =\{2,3,2,3,\dots\}$, or $\eta(g)= 2g-2$. 
\end{rem*}
On the other hand, by \cite[Corollary 2.3]{Cheng-75} we know that 
\[\lambda_i(X_g)\leq \frac{1}{4}+i^2\cdot \frac{16 \pi^2}{\diam^2(X_g)}.\]
By Gauss-Bonnet, $\area(X_g)=4\pi(g-1)$. A simple area argument implies that the diameter $\diam(X_g)\geq C\ln (g)$ for some universal constant $C>0$. So if $\eta(g)$ satisfies $\lim \limits_{g\to \infty} \frac{\eta(g)}{\ln(g)}=0$, we have 
\[\limsup_{g\to \infty}  \sup_{X_g \in \sM_g} \spg_{\eta(g)}(X_g)\leq \frac{1}{4}.\]
Together with Theorem \ref{mt-1} this yields the following direct consequence.
\begin{corollary}\label{mt-2}
If $\eta(g)=o(\ln (g))$, then
\[\lim_{g\to \infty}  \sup_{X_g \in \sM_g} \spg_{\eta(g)}(X_g)= \frac{1}{4}.\]
\end{corollary}

\noindent For $\eta(g)=1$, both Theorem \ref{mt-1} and Corollary \ref{mt-2} are due to Hide--Magee \cite[Corollary 1.3]{HM21}, in which they used a probabilistic method to solve the conjecture (e.g. see \cite{Buser84, BBD88}) that there exists a sequence of closed hyperbolic surfaces with first eigenvalues tending to $\frac{1}{4}$ as the genus goes to infinity.

The following result is important in the proof of Theorem \ref{mt-1}, which we list out for independent interest. The proof is highly motivated by work of Burger--Buser--Dodziuk \cite{BBD88} where they studied the case  when the limiting surface is connected (e.g. see Theorem \ref{BBD-conn}). 
\begin{proposition}[Min-Max Principle]\label{mmp}
Let $X_{g}(0) \in \partial \sM_{g}$ be the limit of a family of Riemann surfaces $\{X_{g}(t)\}$ obtained by pinching certain simple closed geodesics such that $X_{g}(0)$ has $k$ connected components, i.e., $X_{g}(0)=Y_1\sqcup Y_2 \cdots \sqcup Y_k$ where $k\geq 2$. Let $\lambda_{1}(Y_1), \cdots, \lambda_{1}(Y_k)$ be the first non-zero eigenvalue of $Y_1, \cdots, Y_k$ (if $Y_{i}$ has no discrete eigenvalues then denote $\lambda_{1}(Y_{i})=\infty$) and denote $\bar \lambda_{1}(*)=\min\left\{\lambda_{1}(*), \frac{1}{4}\right\}$ for $*=Y_1,\cdots, Y_k$.  Then 
$$
\liminf_{t\rightarrow 0} \lambda_{k}(X_g(t))\geq \min\limits_{1\leq i \leq k} \{\bar \lambda_{1}(Y_i)\}. 
$$
\end{proposition}
\begin{rem*}
Each component $Y_i$ in the proposition above is a complete open hyperbolic surface of finite volume, whose spectrum consists of possibly discrete eigenvalues and the continuous spectrum $[\frac{1}{4},\infty)$. Therefore in the statement above, $\bar \lambda_{1}(Y_i)$ is the non-zero minimum of spectrum of $Y_i$. 
\end{rem*}

\subsection*{Proof sketch of Theorem \ref{mt-1}.} In the proof of Theorem \ref{mt-1}, we will apply Proposition \ref{mmp} to the case when all $\bar \lambda_{1}(Y_i)'s$ are close to $\frac{1}{4}$. The main idea is the following: for each $\eta(g)$ we construct a sequence of genus $g$ surfaces that degenerate into $\eta(g)$ components using only pieces that are known to have the first eigenvalue close to $1/4$. Then by the Min-Max principle, the $\eta(g)$-th eigenvalue of these surfaces will be close to $1/4$. On the other hand, by a result of Schoen--Wolpert--Yau (see~Theorem \ref{SWY-80}) the $(\eta(g)-1)$-th eigenvalue is close to 0. This way we find sequences of surfaces that achieve the spectral gap of $1/4$. For the regime $\eta(g)>g$, the components used in the construction only include the thrice-punctured sphere and a twice-punctured torus. On the other hand, for $\eta(g)\leq g$, the essential components also include a large genus piece that relies on the work of Hide--Magee \cite{HM21}. 

\subsection*{Plan of the paper.} Section \ref{pre} will first discuss properties of the boundary degeneration of the Riemann moduli spaces; then we will provide a review of the background and recent developments on spectral gaps on hyperbolic surfaces, including a list of punctured surface components with eigenvalue bounds which will be used in the degeneration limits. In Section \ref{sec-mmp} we will provide a proof for Proposition \ref{mmp} regarding the Mini-Max Principle for eigenvalues of degenerating hyperbolic surfaces and a few immediate applications. In Section \ref{sec-mt-1} we will complete the proof of Theorem \ref{mt-1}.

\subsection*{Acknowledgement}
The authors would like to thank the anonymous referees for their careful reading and valuable comments for improving the quality of this paper. The first and second named author would like to thank Yuxin He, Yang Shen and Yuhao Xue for helpful discussions on this project. The first named author is partially supported by the NSFC grant No. $12171263$. The third named author is supported by NSF DMS-2041823.

\section{Preliminaries}\label{pre}
\subsection{Boundary of the Riemann moduli spaces}
Denote by $\sM_{g,n}$ the moduli space of hyperbolic surfaces of genus $g$ with $n$ punctures, and $\sM_{g}:=\sM_{g,0}$ the moduli space of compact hyperbolic surfaces with genus $g$. It is well-known that $\dim_{\R}(\sM_{g,n})=6g+2n-6$. In particular, $\sM_{0,3}$ contains only one point represented by the hyperbolic thrice-punctured sphere. The Deligne--Mumford compactification of $\sM_{g,n}$ is by adding nodal surfaces into $\sM_{g,n}$, which is homeomorphic to the completion of $\sM_{g,n}$ endowed with the Weil--Petersson metric. 
Let $\partial \sM_{g,n}$ be the boundary of the Deligne--Mumford compactification of $\sM_{g,n}$. Recall that $\partial \sM_{g,n}$ is stratified, and each stratum of $\psM_{g,n}$ is a product of lower dimensional moduli spaces. Points in $\partial \sM_{g,n}$  are represented by hyperbolic nodal surfaces in $\sM_{g,n}$ (see for example \cite{Masur76} for more detailed description on the completion of $\sM_{g,n}$). Locally the process of pinching a simple closed geodesic into a pair of cusp points can be written with respect to hyperbolic collar coordinates $(\rho, \theta)$: 
where $\ell$ is the length of the central geodesic circle. As $\ell\rightarrow 0$, the hyperbolic neck degenerates into a pair of cusps, which can be seen with the choice of the correct coordinates (see for example~\cite{Masur76, Ji93}). Another way to see this would be using the complex ``plumbing'' coordinates, which we will not . Hyperbolic nodal surfaces are obtained by pinching certain disjoint geodesic circles, and 
we call such family of hyperbolic metrics approaching nodal surfaces a degenerating family (see e.g.~\cite{Wolpert90}, and see Figure \ref{f:df} for an example).

	\begin{figure}
		\centering
		\includegraphics[width=8cm]{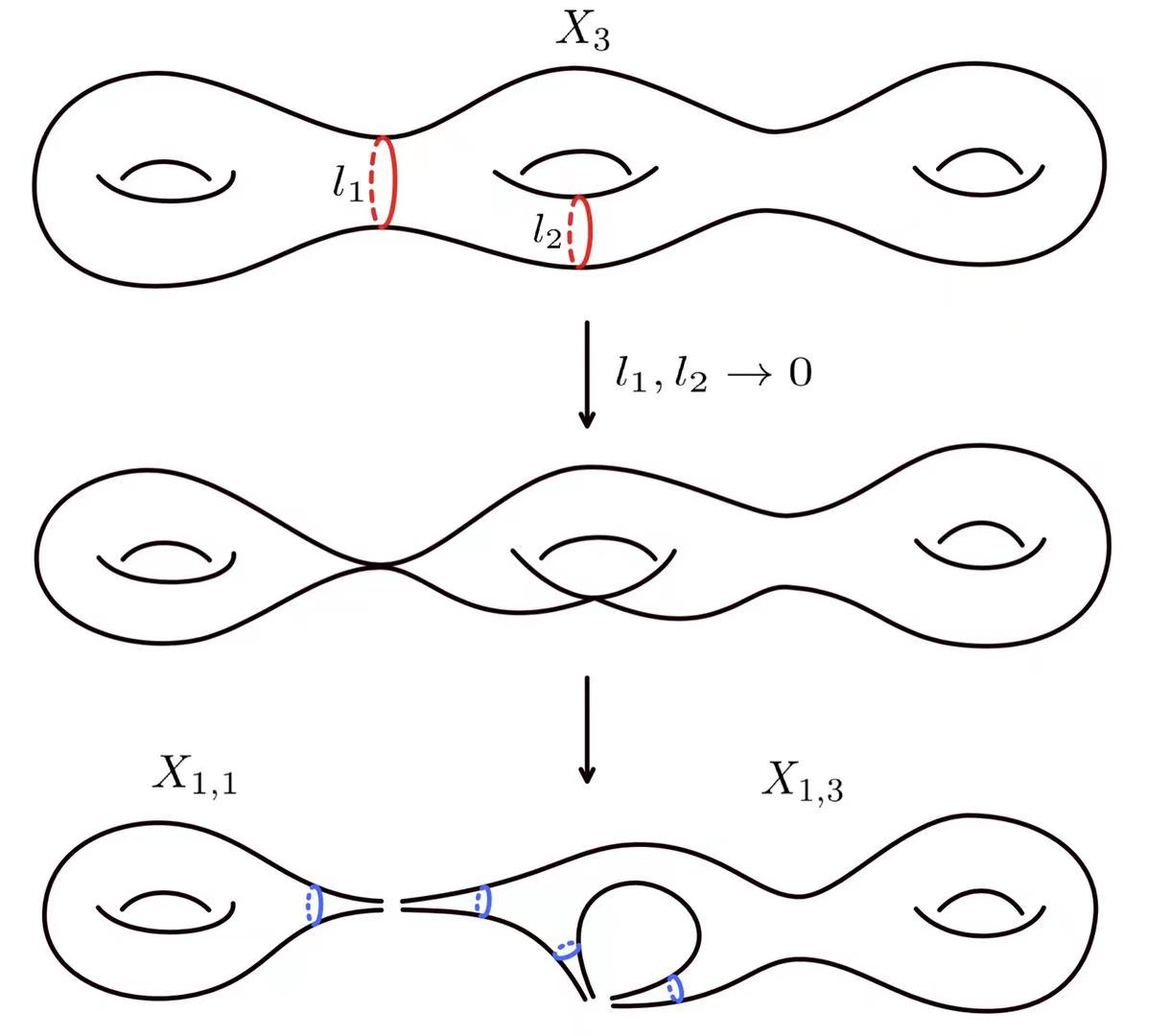}
		\caption{An example of a degenerating family in $\sM_3$ whose limit is $X_{1,1}\sqcup X_{1,3}$ which is disconnected.}
		\label{f:df}
	\end{figure}

We also recall  the  Collar  Lemma on structures  of  disjoint hyperbolic  collars  around short  geodesics, which will be useful later in decomposing the surfaces.  
\begin{lemma}[Collar Lemma]{\cite[Theorem 4.1.1]{Buser10}} \label{cl}
Let $\gamma_1 , \gamma_2, ..., \gamma_m$ be disjoint simple closed geodesics on a closed hyperbolic Riemann surface $X_g$, and $\ell(\gamma_i)$ be the length of $\gamma_i$. Then $m\leq 3g-3$ and we can define the collar of $\gamma_i$ by
	$$T(\gamma_i)=\{x\in X_g; \ \mathrm{dist}(x,\gamma_i)\leq w(\gamma_i)\}$$
	where
	\begin{equation}\label{e:w}
	w(\gamma_i)=\mathop{\rm arcsinh} \frac{1}{\sinh \left(\frac{1}{2}\ell(\gamma_i)\right)}
	\end{equation}
	is the half width of the collar.
	
	Then the collars are pairwise disjoint for $i=1,...,m$. Each $T(\gamma_i)$ is isomorphic to a cylinder $(\rho,\theta)\in [-w(\gamma_i),w(\gamma_i)] \times \mathbb S ^1$, where $\mathbb S ^1 = \mathbb{R} / \mathbb{Z}$, with the metric
\bear\label{collar-metric}	
ds^2=d\rho^2 + \ell(\gamma_i)^2 \cosh^2\rho d\theta^2.
\eear
	And for a point $(\rho,\theta)$, the point $(0,\theta)$ is its projection on the geodesic $\gamma_i$, $|\rho|$ is the distance to $\gamma_i$, $\theta$ is the coordinate on $\gamma_i \cong \mathbb S ^1$.
\end{lemma}

\noindent As the length $\ell(\gamma)$ of the central closed geodesic goes to $0$, the width $ w(\gamma)\sim \ln{\left(\frac{1}{\ell(\gamma)}\right)}$ which tends to infinity. As an easy corollary,
\begin{corollary}\label{cl-d}
For a degenerating family of hyperbolic surfaces $\{X_{g}(t)\}$, the diameter satisfies
$$\diam(X_g(t))\to \infty.$$
\end{corollary}

The following two lemmas  will be useful in the proof of Theorem \ref{mt-1}.

\begin{lemma}\label{d-f-g-2}
For each integer $\eta(g) \in [g-1,2g-2]$ where $g\geq 2$, there exist two non-negative integers $i$ and $j$ such that
\ben
\item $i+j=\eta(g)$;
\item $\underbrace{\sM_{0,3}\times\cdots\times\sM_{0,3}}_{\text{$i$ copies}} \times \underbrace{\sM_{1,2}\times\cdots\times\sM_{1,2}}_{\text{$j$ copies}} \subset \psM_{g}$.
\een
\end{lemma}
\begin{rem*}
Here $i$ and $j$ depend on $g$, and satisfy $i+2j=2g-2$ by the additivity of Euler characteristic. 
\end{rem*}

\bp
If $\eta(g)=2g-2$, the conclusion is obvious by choosing $i=2g-2$ and $j=0$, which is obtained by pinching $3g-3$ disjoint simple closed curves in a closed surface $X_{g}$ of genus $g$. 

Now we assume that $g\leq \eta(g)\leq 2g-3$. Given a closed surface $X_g$ of genus $g$, first one may pinch $X_g$ along $2$ disjoint simple closed curves $\sigma_1$ and $\sigma_2$ such that $X_g\setminus (\sigma_1\cup \sigma_2)$ has two connected components $X_{g_1,2}\sqcup X_{g_2,2}$, where $g_1, g_2$ are two non-negative integers satisfying $g_1+g_2=g-1$. Here we choose
$$
g_{1}=(2g-2)-\eta(g), \ g_{2}=\eta(g)-(g-1).
$$
For the second step, we pinch $X_{g_1,2}$ along $(g_1-1)$ disjoint simple closed curves  $\{\gamma_l\}_{1\leq l\leq g_1-1}$ such that the complement decomposes further into $g_{1}$ components
 $$X_{g_1,2}\setminus \bigcup\limits_{1\leq l \leq g_1-1}\gamma_l= \underbrace{X_{1,2}\sqcup \cdots \sqcup X_{1,2}}_{\text{$g_1$ copies}}.$$
 For $X_{g_2,2}$, one may pinch along $(3g_2-1)$ disjoint simple closed curves  $\{\gamma_m'\}_{1\leq m\leq 3g_2-1}$ such that the complement $$X_{g_2,2}\setminus \bigcup\limits_{1\leq m \leq 3g_2-1}\gamma_m'= \underbrace{X_{0,3}\sqcup \cdots \sqcup X_{0,3}}_{\text{$2g_2$ copies}}.$$ 
Pinching all these simple closed curves during cutting above to $0$, then the conclusion follows since

\begin{equation} \label{e:ij1}
i=2g_2=2\eta(g)-(2g-2) \text{ and } j=g_1=(2g-2)-\eta(g).
\end{equation}
 (For example, see Figure~\ref{f:F1}).

	\begin{figure}[ht]
		\centering
		\includegraphics[width=6cm]{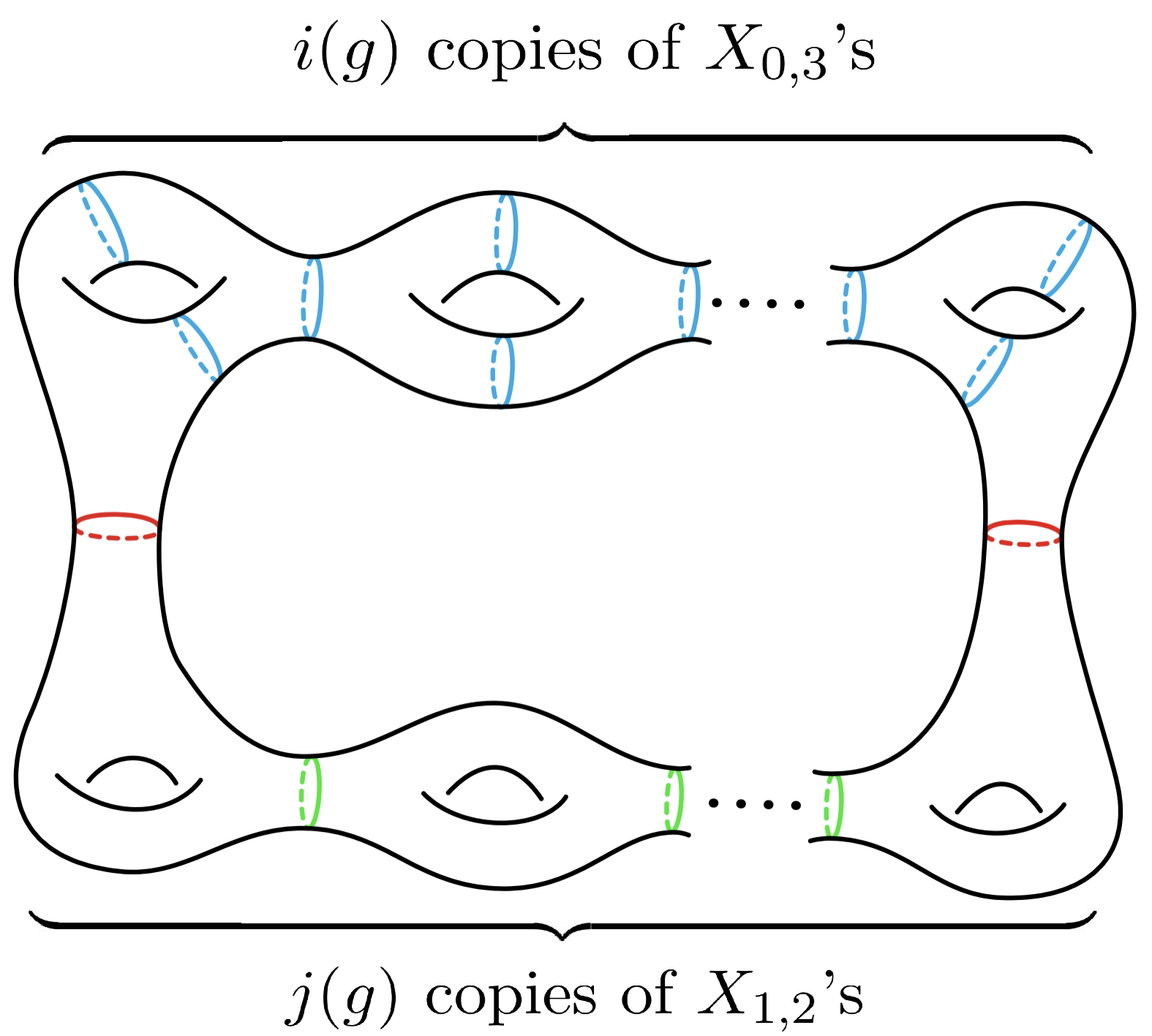}
		\caption{An example of the degeneration of a genus $g$ surface into $i(g)$ copies of $X_{0,3}$'s and $j(g)$ copies of $X_{1,2}$'s by pinching all the simple geodesics marked in the picture}
		\label{f:F1}
	\end{figure}

If $\eta(g)=g-1$, we first pinch $X_g$ along a non-separating simple close curve to get a surface $X_{g-1,2}$. Then in the same way as we did with $X_{g_1,2}$ in the previous case, we pinch $X_{g-1,2}$ along $(g-2)$ disjoint simple closed curves to get $(g-1)$ copies of $X_{1,2}$'s. Then the conclusion follows  where $i=0$ and $j=g-1$. 

Combining the three cases above, the proof is complete.
\ep

\begin{lemma}\label{d-f-g-1}
For each integer $\eta(g) \in [2,g]$ where $g\geq 3$, there exist three non-negative integers $g_1$, $i$ and $j$ such that
\ben
\item $2g_1\geq g-2$;

\item $i+j+1=\eta(g)$;

\item $\underbrace{\sM_{0,3}\times\cdots\times\sM_{0,3}}_{\text{$i(g)$ copies}} \times \underbrace{\sM_{1,2}\times\cdots\times\sM_{1,2}}_{\text{$j(g)$ copies}} \times \sM_{g_1,2}\subset \psM_{g}$.
\een

\end{lemma}
\begin{rem*}
Similar to the previous lemma, $i, j$ and $g_{1}$ depend on $g$. 
By counting the Euler characteristics, those numbers should satisfy $i+2j+2g_1=2g-2$.
\end{rem*}

\bp 
Similar as in the proof of Lemma \ref{d-f-g-2} above we first decompose $X_g$ as $X_g\setminus (\sigma_1\cup \sigma_2)=X_{g_1,2}\sqcup X_{g_2,2}$ for two disjoint simple closed curves $\sigma_1$ and $\sigma_2$ where $g_{1}$ and $g_{2}:=g-1-g_{1}$ will be determined in different cases below. Next we decompose $X_{g_2,2}$ into disjoint union of $i$ copies of $X_{0,3}$'s and $j$ copies of $X_{1,2}$'s to obtain the desired properties. For an illustration, see Figure~\ref{f:F2}.

	\begin{figure}[ht]
		\centering
		\includegraphics[width=7cm]{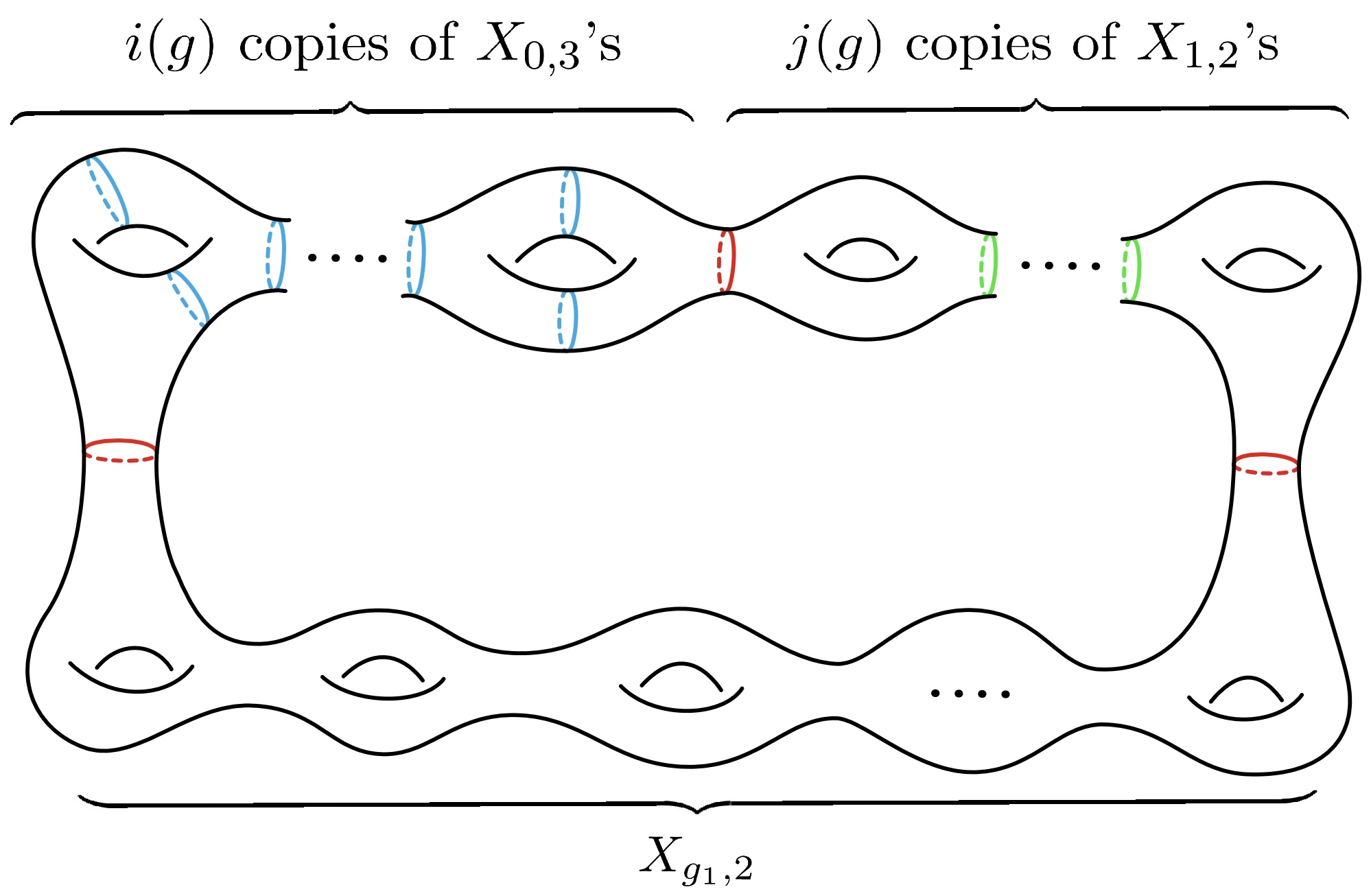}
		\caption{An example of decomposing a surface of genus $g$ into $i$ copies of $X_{0,3}$'s, $j$ copies of $X_{1,2}$'s and a copy of $X_{g_{1},2}$, where $i, j, g_{1}$ are given in the proof of Lemma~\ref{d-f-g-1}.}
		\label{f:F2}
	\end{figure}

The proof contains the following three cases.\vspace{.1in}

 \emph{Case $1$: $2\leq \eta(g)\leq \frac{g}{2}+1$.} The conclusion follows by choosing 
\[i=0, \ j=\eta(g)-1 \ \textit{and} \ g_1=g-\eta(g).\]

 \emph{Case $2$: $\frac{g}{2}+1<\eta(g)\leq g$ and $\eta(g)$ is odd.} The conclusion follows by choosing 
\[i=\eta(g)-1, \ j=0 \ \textit{and} \ g_1=g-\frac{1+\eta(g)}{2}.\]

 \emph{Case $3$: $\frac{g}{2}+1<\eta(g)\leq g$ and $\eta(g)$ is even.} The conclusion follows by choosing 
\[i=\eta(g)-2, \ j=1 \ \textit{and} \ g_1=g-1-\frac{\eta(g)}{2}.\]

The proof is complete.
\ep


\subsection{Eigenvalues of hyperbolic surfaces}
The study of eigenvalues of the Laplacian on hyperbolic surfaces has a long history and has recently seen many progress. For a compact hyperbolic surface, the eigenvalues are discrete. On the other hand, when the hyperbolic surface degenerates to one with cusps, by \cite{LP82} it is known that the spectrum is no longer discrete, rather it consists of a continuous spectrum $[\frac{1}{4}, \infty)$ and (possibly) additional discrete eigenvalues. The study of spectral degeneration has seen many developments, see~\cite{Hejhal90, Ji93, JiZworski, Wolpert87, Wolpert92} for some of the earlier works. 

An eigenvalue of a hyperbolic surface is said to be ``small'' if it is below $1/4$ where the number $1/4$ shows up as the bottom of the continuous spectrum of a hyperbolic surface with cusps. The questions of existence of eigenvalues below $1/4$ for both noncompact and compact hyperbolic surfaces not only arise in the field of spectral geometry, but also has deep relations to number theory regarding arithmetic hyperbolic surfaces, dating back to Selberg's famous $3/16$ theorem~\cite{Selberg65} and we refer to~\cite{GJ78, LRS95, Kim03} for more recent developments.  Regarding the estimates and multiplicity counting of small eigenvalues, the history goes back to McKean~\cite{McKean72, McKean74}, Randol~\cite{Randol74}, Buser~\cite{Buser-82, Buser84}. Recently there has been many developments, see~\cite{BM01, OR09, Buser10, Mondal15, BMM16, BMM17, BMM18} and references therein for more complete reference. Among these there are two classical results of particular relevance to our current work. One is regarding bounds of eigenvalues on degenerating hyperbolic surfaces by Schoen--Wolpert--Yau~\cite{SWY80}:
\begin{theorem}[Schoen--Wolpert--Yau '80]\label{SWY-80}
For any compact hyperbolic surface $X_{g}$ of genus $g$ and integer $i\in (0, 2g-2)$, the $i$-th eigenvalue satisfies
$$
\alpha_{i}(g)\cdot \ell_{i} \leq \lambda_{i} \leq \beta_{i}(g) \cdot \ell_{i}
$$
and 
$$
\alpha(g) \leq \lambda_{2g-2}
$$
where $\alpha_{i}(g)>0$ and $\beta_{i}(g)>0$ depend only on $i$ and $g$, $\alpha(g)>0$ depends only on $g$, and $\ell_{i}$ is the minimal possible sum of the lengths of simple closed geodesics in $X_{g}$ which cut $X_{g}$ into $i+1$ connected components. 
\end{theorem}
\noindent Dodziuk and Randol in~\cite{DR86} gave an alternative proof of Theorem \ref{SWY-80}, and one may also see Dodziuk--Pignataro--Randol--Sullivan \cite{DPRS85} on similar results for Riemann surfaces with punctures. It was proved by Otal--Rosas \cite{OR09} that the constant $\alpha(g)$ can be optimally chosen to be $\frac{1}{4}$. For large genus $g$, it was recently proved by the first named author and Xue \cite{WX21-L1, WX18} that up to multiplication by a universal constant, $\alpha_1(g)$ can be optimally chosen to be $\frac{1}{g^2}$.

The other result that is relevant is~\cite[Theorem 2.1]{BBD88} regarding the first eigenvalue when the limiting degenerating surface is connected:
\bt[Buser--Burger--Dodziuk '88]\label{BBD-conn}
Let $\{X_g(t)\}\subset \sM_g$ such that $Y=\lim \limits_{t\to 0}X_g(t)\in \partial \sM_g$ is connected. Denote $\lambda_{1}(Y)$ the first nonzero eigenvalue of $Y$ (if $Y$ has no discrete eigenvalues we denote $\lambda_{1}(Y)=\infty$). Then
\[\limsup_{t\to 0}\lambda_1(X_g(t))\geq \bar\lambda_1(Y)=\min\left\{\lambda_1(Y),\frac{1}{4} \right\}.\]
\et
In Section 3 we will give a similar description of $\lambda_k(X_g(t))$ when the limiting surface has $k$ connected components.

Another related direction in this topic is to understand how the genus of the hyperbolic surface, in particular when $g\rightarrow \infty$, affects the eigenvalues via different models of random hyperbolic surfaces. Brooks--Makover \cite{BM04} gave a uniform lower bound on the first spectral gap for their combinatorial model of random surfaces by gluing hyperbolic ideal triangles. In terms of Weil--Petersson random closed hyperbolic surfaces, Mirzakhani \cite{Mirz13} showed that the first eigenvalue goes above $0.0024$ with probability one as $g \rightarrow \infty$. Recently, the first named author and Xue \cite{WX21} improved this lower bound $0.0024$ to be $\frac{3}{16}-\epsilon$, which was also independently obtained by Lipnowski and Wright in \cite{LipWright}. One may also see Hide \cite{Hide21} for similar results on  Weil--Petersson random punctured hyperbolic surfaces, and see Monk \cite{Monk21} for related results. Recently there has also been many exciting development in the case of random covers of both compact and noncompact hyperbolic surfaces, see~\cite{MP20, MNP20, MN20, MN21}. For example, Magee--Naud--Puder \cite{MNP20} showed that a generic covering of a hyperbolic surface has relative spectral gap of size $\frac{3}{16}-\epsilon$, which was improved to $\frac{1}{4}-\epsilon$ by Hide-Magee \cite{HM21} for random covers of punctured hyperbolic surfaces. As an important application, Hide and Magee in \cite{HM21} proved that $\lim\limits_{g\to \infty} \sup\limits_{X_g \in \sM_g} \lambda_1(X_g)= \frac{1}{4}$. This result  provides major inspiration for our current paper. 

One major ingredient of our proof is the existence of punctured surfaces with first eigenvalue close to $1/4$. We summarize those components in the two theorems below:
\begin{theorem}\label{t:0312}
\begin{enumerate}
\item $\lambda_1(X_{0,3})\geq\frac{1}{4}$; 
\item $\mathrm{(Mondal \textit{'}15)}$ There exists a surface $X_{1,2}\in \sM_{1,2}$ such that $\lambda_1(X_{1,2})\geq\frac{1}{4}.$
\end{enumerate}
\end{theorem}
\begin{proof}
The first item is well-known, see for example \cite{OR09} or \cite{BMM16}. The existence of the second item was proved by Mondal in~\cite[Theorem 1.3]{Mondal15}. 
\end{proof}

The third component is from the recent breakthrough by Hide-Magee \cite{HM21}. They use probabilistic method to show that for any $\epsilon>0$, there exists an integer $\delta(\epsilon)>0$ only depending on $\epsilon$ such that for all $g>\delta(\epsilon)$ there exists a $2g$-cover $\mathcal{X}$ of $X_{0,3}$ such that \[\bar\lambda_1(\mathcal{X})=\min\left\{\lambda_1(\mathcal{X}), \frac{1}{4}\right\}>\frac{1}{4}-\epsilon.\]
It is not hard to see that $\mathcal{X}$ must have even number of punctures because the Euler characteristic of $\mathcal{X}$ is equal to $-2g$ which is even. Then one may apply the Handle Lemma of Burger--Buser--Dodziuk \cite{BBD88} (or see \cite[Lemma 1.1]{BM01}) to get
\bt \label{hm-o}
For any $\epsilon>0$ and large enough $g>0$, there exists a hyperbolic surface $\mathcal{X}_{g,2}\in \sM_{g,2}$ such that 
\[\bar \lambda_1(\mathcal{X}_{g,2})=\min\left\{\lambda_1(\mathcal{X}_{g,2}), \frac{1}{4}\right\}>\frac{1}{4}-\epsilon.\]
\et
\bp
For completeness we sketch an outline of the proof. Suppose by contradiction there exists a constant $\epsilon_0>0$ such that
\begin{equation}\label{as-up}
\liminf_{g\to \infty} \sup_{X\in \sM_{g,2}}\lambda_1(X)\leq \frac{1}{4}-\epsilon_0.
\end{equation}
It follows by~\cite{HM21} of Hide-Magee that for any $\epsilon>0$ and large enough $g$, there exists a $2g$-cover $\mathcal{X}$ of $X_{0,3}$ such that $\bar\lambda_1(\mathcal{X})=\min\left\{\lambda_1(\mathcal{X}), \frac{1}{4}\right\}>\frac{1}{4}-\epsilon$. Since the Euler characteristic $\chi(\mathcal{X})=-2g$ is even, one may assume that $\mathcal{X}$ has even number of cusps. As in \cite{BBD88} we can construct a family of hyperbolic surfaces $\{X_{g,2}(t)\}\subset \sM_{g,2}$ such that $$\lim \limits_{t\to 0}X_{g,2}(t)= \mathcal{X} \in \partial \sM_{g,2}.$$ By \cite{LP82} of Lax-Phillips we know that for a hyperbolic surface with cusps, the spectrum below $1/4$ is discrete and only contains eigenvalues. By \eqref{as-up}, for some large $g$ one may assume that $\phi_t$ is the first eigenfunction on $X_{g,2}(t)$ with $\Delta \phi_t=\lambda_1(X_{g,2}(t))\cdot \phi_t$ on $X_{g,2}(t)$. Then one may apply the Handle Lemma of Burger--Buser--Dodziuk \cite{BBD88} (or see \cite[Lemma 1.1]{BM01}) to obtain 
\[\limsup_{t\to 0}\lambda_1(X_{g,2}(t))\geq\bar\lambda_1(\mathcal{X})=\min\left\{\lambda_1(\mathcal{X}), \frac{1}{4}\right\}>\frac{1}{4}-\epsilon \]
which is a contradiction to \eqref{as-up} since $\epsilon>0$ can be chosen to be arbitrarily small.
\ep

\section{Eigenvalues on a family of degenerating Riemann surfaces}\label{sec-mmp}
In this section we will prove the following Min-Max principle:
\begin{proposition}[Min-Max Principle, same as Proposition~\ref{mmp}]
Let $X_{g}(0) \in \partial \sM_{g}$ be the limit of a family of Riemann surfaces $\{X_{g}(t)\}$ obtained by pinching certain simple closed geodesics such that $X_{g}(0)$ has $k$ connected components, i.e., $X_{g}(0)=Y_1\sqcup Y_2 \cdots \sqcup Y_k$ where $k\geq 2$. Let $\lambda_{1}(Y_1), \cdots, \lambda_{1}(Y_k)$ be the first non-zero eigenvalue of $Y_1, \cdots, Y_k$ (if $Y_{i}$ has no discrete eigenvalues then denote $\lambda_{1}(Y_{i})=\infty$) and denote $\bar \lambda_{1}(*)=\min\left\{\lambda_{1}(*), \frac{1}{4}\right\}$ for $*=Y_1,\cdots, Y_k$.  Then 
$$
\liminf_{t\rightarrow 0} \lambda_{k}(X_g(t))\geq \min\limits_{1\leq i \leq k} \{\bar \lambda_{1}(Y_i)\}. 
$$
\end{proposition}

To prove the theorem, we will start by discussing the subsequence limits of eigenfunctions. 
Denote $\phi_{t}\in C^{\infty}(X_{g}(t))$ (one of) the normalized eigenfunction corresponding to $\lambda_{k}(X_{g}(t))$, \emph{i.e.},
$$
\Delta_{X_{g}(t)} \phi_{t}=\lambda_{k}(X_{g}(t))\cdot \phi_t \ \emph{and} \ \int_{X_{g}(t)} |\phi_{t}|^{2} \mathrm{dVol}_{X_{g}(t)}=1.
$$


By~\cite[Corollary 2.3]{Cheng-75} we know that for any compact hyperbolic surface $X$ there is an upper bound
\[\lambda_k(X)\leq \frac{1}{4}+k^2\cdot \frac{16 \pi^2}{\diam^2(X)}.\]
Note that $\diam(X_{g}(t))\rightarrow \infty$ as $t\rightarrow 0$ by Corollary \ref{cl-d} for any family of degenerating hyperbolic surfaces $\{X_{g}(t)\}$ as described in the proposition above. This gives that for any fixed $k\geq 1$, 
\begin{equation}\label{up-1-4}
\liminf_{t\rightarrow 0} \lambda_{k}(X_{g}(t)) \leq \limsup_{t\rightarrow 0} \lambda_{k}(X_{g}(t))\leq \frac{1}{4}.  
\end{equation}

On the other hand, by Theorem~\ref{SWY-80} we know that the lowest $k-1$ eigenvalues of $X_{g}(t)$ go to 0 when the degenerating limit has $k$ components, while the $k$-th eigenvalue $\lambda_{k}(X_{g}(t))$ stays bounded away from 0. Therefore
\begin{equation}
\liminf_{t\rightarrow 0}\lambda_{k}(X_{g}(t)) >0.
\end{equation}

Now consider 
\begin{equation}\label{e:liminf}
\lambda_{k}(0):=\liminf \limits_{t\rightarrow 0}\lambda_{k}(X_{g}(t)).
\end{equation} 
By the discussion above we know that
\begin{equation}\label{e:lambda0}
0<\lambda_k(0)\leq \frac{1}{4}.
\end{equation}

By the Collar Lemma~\ref{cl}, each $X_{g}(t)$ can be decomposed  into a number of disjoint degenerating hyperbolic necks  and a compact part (which has possibly several connected components). The width of each hyperbolic neck is determined by the central shrinking geodesic $\gamma$ and can be chosen as $(w(\gamma)-1)$ for example, where $w(\gamma)$ is given in~\eqref{e:w}. For the degenerating family $\{X_{g}(t)\}$ with $N$ shrinking geodesics $\{\gamma_{m}(t)\}_{m=1}^{N}$, we denote the width of each hyperbolic neck by the following $N$-tuple:
$$
\vec w:=\left(w(\gamma_{1}(t))-1, w(\gamma_{2}(t))-1, \dots, w(\gamma_{N}(t))-1\right).
$$
Note that $\vec w$ depends on $t$, and each entry in $\vec w$ goes to $\infty$ as $t$ goes to 0. Geometrically each hyperbolic neck degenerates into a pair of cusps. We remark here that in the definition of $\vec w$,  the choice $(w(\gamma)-1)$ is for convenience and can be replaced by $(w(\gamma)-c)$ with any $c>0$.

For any $X_{g}(t)$, we denote the union of all $N$ hyperbolic necks as $C_{\vec w}(t)$. In local hyperbolic geodesic coordinates given by $d\rho^{2}+\ell^{2}\cosh^{2} \rho d\theta^{2}$ where $\ell$ is the length of the central geodesic circle $\gamma_{i}$,
\begin{equation}\label{def-c}
C_{\vec w}(t)=\bigcup_{m=1}^{N}\left\{(\rho, \theta): \ 0\leq |\rho|\leq w(\gamma_{m}(t))-1\right\}.
\end{equation}
In addition, we also denote the union of all ``shells'' near the collars by
\begin{equation}\label{def-s}
S_{\vec w}(t)=\bigcup_{m=1}^{N} \left\{(\rho, \theta): \ w(\gamma_{m}(t))-1\leq |\rho| \leq w(\gamma_{m}(t))\right\}.
\end{equation}
Then it follows by the Collar Lemma that all such collar neighborhoods (resp. shells) are disjoint (see Figure \ref{f:cs} for an illustration of collars and shells for an example).

	\begin{figure}
		\centering
		\includegraphics[width=10cm]{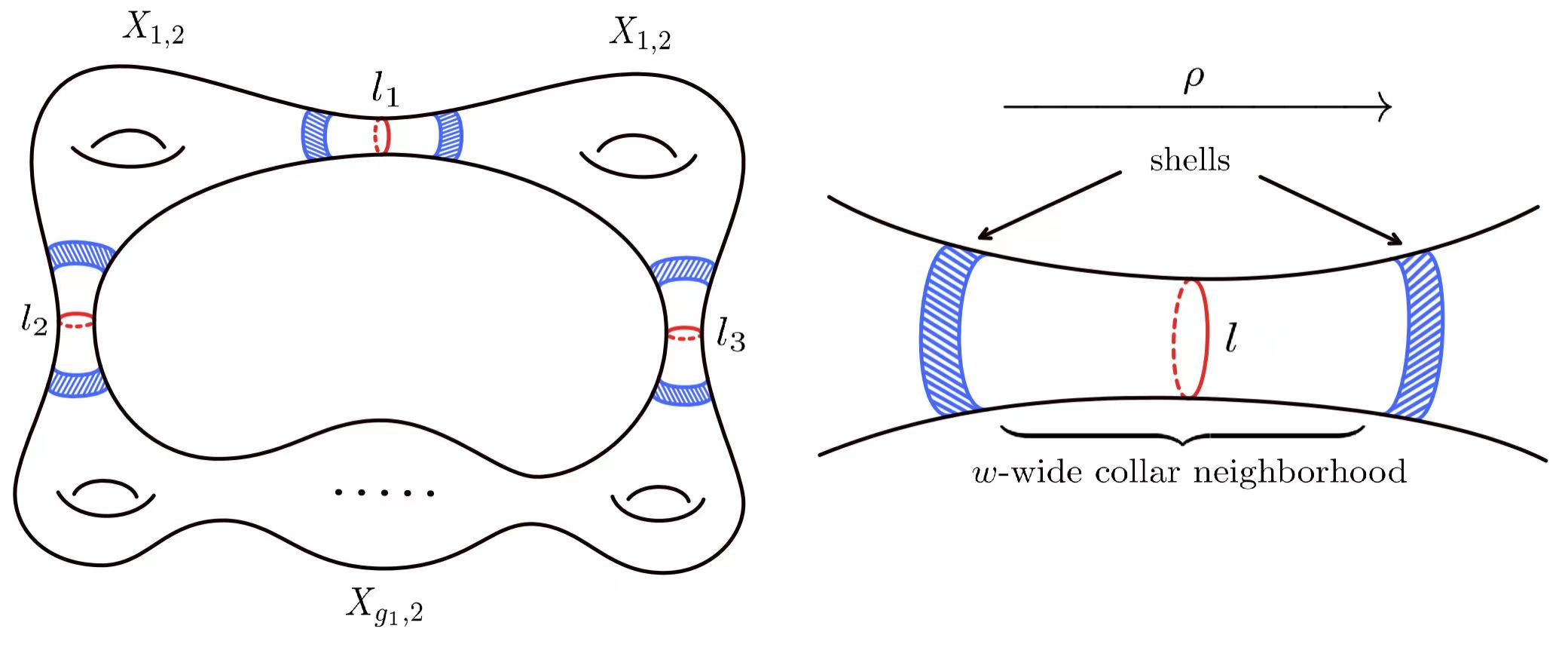}
		\caption{An example of collar neighborhoods and shells.}
		\label{f:cs}
	\end{figure}

Denote the compact part as $F_{\vec w}(t)=X_{g}(t)\setminus C_{\vec w}(t)$. The compact area and nodal degeneration area are grafted together~\cite{Wolpert90, MZ18, MZ19}.  For small $t$, $\{F_{\vec w}(t)\}$ are all diffeomorphic. In particular, the metric on $F_{\vec w}(t)$ can be written by $e^{2u_{t}}g_{0}$ where $g_{0}$ is the metric on $F_{\vec w}(0)$ and $u_{t}$ are polyhomogeneous and uniformly bounded in all derivatives~\cite{MZ19}. That is, we can write the diffeomorphism $D_{t}: F_{\vec w}(t)\rightarrow F_{\vec w}(0)$ such that $g_{t}=D_{t}^{*}g_{0}$ and $D_{t}$ are uniformly bounded. From now on when we consider the convergence of eigenfunctions $\phi(t)$ on $X_{g}(t)$, the functions are all defined on $X_{g}(0)$ via pullback $(D_{t}^{-1})^{*}\phi(t)$, see~\cite{Wolpert92} for similar approaches. See also another related approach via universal covers in~\cite{BBD88}.

Now take a sequence of metrics such that the corresponding eigenvalues approaches $\lambda_k(0)$ which is the defined in~\eqref{e:liminf}. Denote the sequence by $\{X_{g}(t_{i})\}_{i=1}^{\infty}$. By definition,
$$
\lim_{i\rightarrow \infty} t_{i}=0 \ \textit{and} \ \lim_{i\rightarrow \infty} \lambda_{k}\left(X_{g}(t_{i})\right) = \lambda_{k}(0). 
$$
Denote the corresponding eigenfunction on $X_{g}(t_{i})$ by $\phi_{t_{i}}$, and below we discuss the convergence of the sequence of functions $\{\phi_{t_{i}}\}_{i=1}^{\infty}$. One key ingredient is the following Sobolev--G\r{a}rding Inequality on the compact part $F_{\vec w}(t)$. Denote by $\mathrm{inj}(\cdot)$ the injectivity radius function.  Denote by $\nabla^{j} \phi_{t_{i}}$ the $j$-th covariant derivative of $\phi_{t_{i}}$ where $j \in \mathbb{N}$. Then we have the following
\begin{lemma}
For any $x\in F_{\vec w}(t)$, $j\in \mathbb{N}$ and $r<\mathrm{inj}(F_{\vec w}(t))$, there exist a constant $c_{r,j}>0$ and an integer $N_{j}>0$ independent of $x$ such that we have the pointwise bound for any $j$-th derivative:
\begin{equation}\label{e:sg}
|\nabla^{j} \phi_{t}(x)|\leq c_{r,j} \sum_{\ell=0}^{N_{j}} \|\Delta_{X_{g}(t)}^{\ell}\phi_{t}\|_{L^{2}(B_{r}(x))}.
\end{equation}
\end{lemma}
\begin{proof}
This equality was shown in~{\cite[Theorem 2.1]{BBD88}}. The inequality is from the combination of Sobolev inequality and G\r{a}rding inequality, for example see~\cite{BJS}.
\end{proof}
 
With the above inequality we have the following uniform bound on $\{\phi_{t_{i}}\}_{i=1}^{\infty}$ and their derivatives. \begin{lemma}
For any $j\in \mathbb{N}$, 
 $\{\nabla^{j} \phi_{t_{i}}\}_{i}$ is uniformly bounded on any compact set of $X_{g}(0)$. 
\end{lemma}
\begin{proof}	
Using~\eqref{e:sg} in the previous lemma and that $\Delta  \phi_{t}=\lambda_{k}(t)\phi_{t}$, $0<\lambda_{k}(t)<\frac{1}{3}$, we have
$$
|\nabla^{j} \phi_{t}(x)| \leq c_{r,j}\sum_{\ell=0}^{\infty}\left(\frac{1}{3}\right)^{\ell} \|\phi_{t}\|_{L^{2}(X_{g}(t))}\leq 2c_{r,j}
$$
where the bound is independent of $x$.
Hence all derivatives of $\phi_{t}$ (in particular the sequence $\{\phi_{t_i}\}$) are uniformly bounded.
\end{proof}

\begin{lemma}
There exists a subsequence of $\phi_{t_i}$ (denoted by $\phi_{i}$) and $\phi_{0} \in H^{1}(X_{g}(0))$ such that any derivatives satisfy
$$
\nabla^{j} \phi_{i}\rightarrow \nabla^{j} \phi_{0}
$$
uniformly on connected compact set of $X_{g}(0)$. 
\end{lemma}
\begin{proof}
Viewing $\{\phi_{t}\}$ as functions on $F_{0}$ where $F_0$ is any connected compact set of $X_g(0)$, by the previous lemma we have uniform boundedness of $\phi_{t}$ and all their derivatives. Hence by Arzela-Ascoli diagonal argument there exists a subsequence $\phi_{i}$ such that the function and its derivative converges uniformly on any compact set. 
\end{proof}
By the convergence above we have
$$
\int_{X_{g}(0)} |\phi_{0}|^{2}\leq 1, \ \int_{X_{g}(0)} |\nabla \phi_{0}|^{2}\leq 1
$$
and 
$$
\Delta_{X_{g}(0)} \phi_{0}=\lambda_{k}(0) \cdot \phi_{0}.
$$

Now we show the following statement regarding the limit $(\lambda_{k}(0), \phi_{0})$. The argument is similar to~\cite[Lemma 9]{WX21-L1} and~\cite[Lemma 3.3]{DPRS85}.
\begin{proposition}
The limit $(\lambda_{k}(0), \phi_{0})$ must satisfy one of the two conditions:
\begin{enumerate}
\item $\phi_{0}$ is an eigenfunction of $\Delta_{X_{g}(0)}$ and also restricts to at least one of the components $Y_{k}$ as an eigenfunction; or
\item $\phi_{0}=0$ everywhere on $X_{g}(0)$ and $\lambda_{k}(0)=\frac{1}{4}$.
\end{enumerate}
\end{proposition}
\begin{proof}	
If $\phi_{0}$ is not 0 everywhere, then $\phi_{0}\in H^{1}(X_{g}(0))$  and is an eigenfunction. In particular, it must restrict to a non-zero function on at least one component of $X_g(0)$. 

Otherwise suppose $\phi_{0}=0$ everywhere on $X_{g}(0)$, that is, $\phi_{i}\rightarrow 0$ pointwise everywhere. Then following a similar argument as in~\cite[Lemma 9]{WX21-L1} or ~\cite[Lemma 3.3]{DPRS85}, we can show that $\lambda_k(0)\geq \frac{1}{4}$. For completeness we write out the proof in detail here. 

Recall the definition of collar and shell on hyperbolic necks in \eqref{def-c} and \eqref{def-s}. Similar to the definition above, we denote $C_{\vec w}(i)$ is the union of $\vec w$-wide collar neighborhoods near all degenerating geodesic circles on $X_{g}(t_i)$, and $S_{\vec w}(i)$ the union of the ``shells''.  To simplify the argument below, we also denote by $C_{i,m}$ (resp. $S_{i,m}$) the individual hyperbolic neck (resp. shell) with central geodesic circle $\gamma_{m}(i)$ where $1\leq m\leq N$, and denote the corresponding width $w_{i,m}:=w(\gamma_{m}(i))-1$. Hence
$$
C_{\vec w}(i)=\cup_{m=1}^{N}C_{i,m}, \ S_{\vec w}(i)=\cup_{m=1}^{N}S_{i,m}.
$$

Fix any $\epsilon\in (0,1)$ and $\delta \in (0, 1/16)$. We denote $c=1-\epsilon$. Since $\phi_{i}$ converges to 0 uniformly on any compact set, there exists $N_0\in \mathbb{N}$ such that for any $i>N_0$ we have
$$
\int_{C_{\vec w}(i)}|\phi_{i}|^{2} \geq c>0, \ \int_{S_{\vec w}(i)} |\phi_{i}|^{2}<\delta c \ \textit{and} \ \int_{S_{\vec w}(i)} |\nabla\phi_{i}|^{2}<\delta c. 
$$
Define a new function on $C_{\vec w}(i)\cup S_{\vec w}(i)$ by the following: 
$$
\Phi_{i}:=\left\{ 
\begin{array}{ll}
\phi_{i}, & |\rho|\leq w_{i,m};\\
(w_{i,m}+1-|\rho|)\phi_{i}, & w_{i,m}\leq |\rho|\leq w_{i,m}+1.
\end{array}
\right.
$$
Then $\Phi_{i}$ gives a function in $H_{0}^{1}(C_{\vec w}(i)\cup S_{\vec w}(i))$ with $\Phi_{i}|_{\partial(C_{\vec w}(i)\cup S_{\vec w}(i))}=0$. Therefore by applying~\cite[Lemma 7]{WX21-L1} to a union of hyperbolic collars we have
$$
\int_{C_{\vec w}(i)\cup S_{\vec w}(i)} |\nabla \Phi_{i}|^{2} >\frac{1}{4} \int_{C_{\vec w}(i)\cup S_{\vec w}(i)} |\Phi_{i}|^{2}. 
$$
On the other hand we have 
$$
\begin{aligned}
\int_{S_{\vec w}(i)}|\nabla \Phi_{i}|^{2}  &=\sum_{m=1}^{N}\int_{S_{i,m}}|\nabla\big((w_{i,m}+1-|\rho|) \phi_{i}\big)|^{2}\\
&=\sum_{m=1}^{N}\int_{S_{i,m}} \left|\nabla(w_{i,m}+1-|\rho|) \cdot \phi_i+(w_{i,m}+1-|\rho|) \cdot \nabla \phi_i\right|^{2}\\
&\leq \sum_{m=1}^{N}\int_{S_{i,m}} \left(|\phi_i|+(w_{i,m}+1-|\rho|) \cdot |\nabla \phi_i|\right)^{2}\\
&\leq 2\sum_{m=1}^{N}\int_{S_{i,m}}|\phi_i|^{2} + 2\sum_{m=1}^{N}\int_{S_{i,m}}|\nabla \phi_i|^{2}\\
&\leq 4\delta c.
\end{aligned}
$$
Therefore for any $i>N_0$ we have
$$
\begin{aligned}
\int_{C_{\vec w}(i)}|\nabla\phi_{i}|^{2}&= \int_{C_{\vec w}(i)}|\nabla\Phi_{i}|^{2}=\int_{C_{\vec w}(i)\cup S_{\vec w}(i)}|\nabla\Phi_{i}|^{2} -\int_{S_{\vec w}(i)}|\nabla\Phi_{i}|^{2}\\
&\geq  \frac{1}{4} \int_{C_{\vec w}(i)\cup S_{\vec w}(i)} |\Phi_{i}|^{2} -\int_{S_{\vec w}(i)}|\nabla\Phi_{i}|^{2}\\
& \geq \frac{1}{4} \int_{C_{\vec w}(i)} |\phi_{i}|^{2} -\int_{S_{\vec w}(i)}|\nabla\Phi_{i}|^{2}\\
&\geq \frac{1}{4}c -4\delta c\\
&=\frac{1-16\delta}{4}(1-\epsilon)
\end{aligned}
$$
which implies that 
$$\lambda_{k}(X_g(t_i))= \frac{\int_{X_g(t_i)}|\nabla\phi_{t_i}|^{2}}{\int_{X_{g}(t_i)}|\phi_{i}|^{2}}\geq \frac{\int_{C_{\vec w}(i)}|\nabla\phi_{i}|^{2}}{\int_{X_{g}(t_i)}|\phi_{i}|^{2}}\geq \frac{1-16\delta}{4}(1-\epsilon).$$
Since this argument holds for any $\epsilon\in (0,1)$ and $\delta\in (0, 1/16)$, we have that 
$$
\lambda_{k}(0)=\liminf \limits_{i\to \infty} \lambda_{k}(X_g(t_i))\geq \frac{1}{4}. 
$$
On the other hand $\lambda_{k}(0)\leq \frac{1}{4}$ by~\eqref{e:lambda0}, therefore we have $\lambda_{k}(0)=\frac{1}{4}$. 
 \end{proof}

Now we are ready to prove Proposition~\ref{mmp}.
\begin{proof}[Proof of Proposition~\ref{mmp}]
By the previous proposition, either $\lambda_{k}(0)=\lambda_{1}(Y_{i})$ for at least one of the components $Y_{i}$, or $\lambda_{k}(0)=\frac{1}{4}$, therefore we obtain
$$
\lambda_{k}(0)\geq \min_{1\leq i\leq k} \left\{\min\left\{\lambda_{1}(Y_{i}), \frac{1}{4}\right\}\right\} 
$$
as desired.
\end{proof}

We enclose in this section the following result, which is an easy application of Proposition~\ref{mmp}.
\begin{proposition}\label{mmp-c}
Let $X_{g}(0) \in \partial \sM_{g}$ be the limit of a family of Riemann surfaces $\{X_{g}(t)\}\subset \sM_g$ by pinching certain simple closed geodesics such that $X_{g}(0)$ has $k$ connected components, i.e., $X_{g}(0)=Y_1\sqcup Y_2 \cdots \sqcup Y_k$ for some $k\geq 2$. Assume in addition that $\bar\lambda_1(Y_i)=\min\left\{\lambda_1(Y_i), \frac{1}{4}\right\}\geq \frac{1}{4}$ for all $1\leq i \leq k$ where $\lambda_{1}(Y_i)$ is the first non-zero eigenvalue of $Y_i$.  Then 
$$
\lim_{t\rightarrow 0} \lambda_{k}(X_g(t))=\frac{1}{4}. 
$$
\end{proposition}
\bp
From \eqref{up-1-4} we have that
\[\limsup_{t\rightarrow 0} \lambda_{k}(X_g(t))\leq\frac{1}{4}. \]
On the other hand, it follows by Proposition \ref{mmp} that
\[\liminf_{t\rightarrow 0} \lambda_{k}(X_g(t))\geq \min_{1\leq i \leq k} \left\{ \min\left\{\lambda_1(Y_i), \frac{1}{4}\right\}\right\}=\frac{1}{4}. \]

Then the conclusion immediately follows.
\ep

We now prove spectral gaps can be arbitrarily close to $0$ by using this result.
Recall that for all $i\geq 1$ and $X_g\in \sM_g$, the $i$-th spectral gap $\spg_i(X_g)$ of $X$ is defined as $$\spg_i(X_g)\overset{\text{def}}{=}\lambda_i(X_g)-\lambda_{i-1}(X_g).$$
We prove the following:
\begin{proposition}\label{spg=0}
For all $i\geq 1$, $$\inf \limits_{X_g \in \sM_g} \spg_i(X_g)=0.$$
\end{proposition}
\bp
We split the proof into the following three cases.
\vspace{.1in}

\emph{Case-$1$: $1\leq i \leq 2g-3$.} 

One may choose a closed hyperbolic surface $\mathcal{X}_g \in \sM_g$ which is close enough to the maximal nodal surface $\underbrace{X_{0,3}\sqcup\cdots \sqcup X_{0,3}}_{\text{$2g-2$ copies}} \in \partial \sM_g$, then $\lambda_i(\mathcal{X}_g)$ is close to $0$ by Theorem~\ref{SWY-80}. So the conclusion follows for this case.
\vspace{.15in}

\emph{Case-$2$: $i = 2g-2$.}

Let $Z_{1,2}\in \sM_{1,2}$ such that $\bar\lambda_1(Z_{1,2})=\min\left\{\frac{1}{4}, \lambda_1(Z_{1,2})\right\}\geq \frac{1}{4}$ by Theorem~\ref{t:0312}. Recall $\lambda_1(X_{0,3})\geq\frac{1}{4}$ from the same theorem. Let $\{X_g(t)\}\subset \sM_g$ be a family of hyperbolic surfaces such that 
\[\lim \limits_{t\to 0}X_g(t)=\underbrace{X_{0,3}\sqcup\cdots\sqcup X_{0,3}}_{\text{$2g-4$ copies}} \sqcup Z_{1,2}\in \partial \sM_g.\]
Then it follows from Proposition \ref{mmp-c} that
\[\lim \limits_{t\to 0}\lambda_{2g-3}(X_g(t))=\frac{1}{4}.\]
Meanwhile, by \cite[Theorem 2]{OR09} we know that $$\lambda_{2g-2}(X_g(t))\geq \frac{1}{4}.$$ Since $\diam(X_g(t))\to \infty$ as $t\to 0$, by \cite[Corollary 2.3]{Cheng-75} we have that 
\[\limsup \limits_{t\to 0}\lambda_{2g-2}(X_g(t))\leq \frac{1}{4}.\]
Thus, we have
\[\lim \limits_{t\to 0}\lambda_{2g-2}(X_g(t))= \frac{1}{4}.\]
Then the conclusion also follows for this case because
\[\inf \limits_{X_g \in \sM_g} \spg_{2g-2}(X_g)\leq \lim \limits_{t\to 0}\spg_{2g-2}(X_g(t))=0.\]
\vspace{.05in}

\emph{Case-$3$: $i>2g-2$.} 

Let $\{Y_g(t)\}\subset \sM_g$ be a family of hyperbolic surfaces such that 
\[\lim \limits_{t\to 0}Y_g(t)\in \partial \sM_g.\]
Similar as in Case-$2$, by \cite[Theorem 2]{OR09} of Otal--Rosas and \cite[Corollary 2.3]{Cheng-75} of Cheng we have
\[\lim \limits_{t\to 0}\lambda_{i}(Y_g(t))= \frac{1}{4} \ \emph{and} \ \lim \limits_{t\to 0}\lambda_{i-1}(Y_g(t))= \frac{1}{4}.\]
This implies  $\inf \limits_{X_g \in \sM_g} \spg_i(X_g)=0$ for all $i>2g-2$.
\vspace{.1in}

The proof is complete.
\ep

\section{Proof of Theorem \ref{mt-1}}\label{sec-mt-1}


Now we are ready to prove Theorem \ref{mt-1}.
\bt[$=$Theorem \ref{mt-1}]\label{mt-1-0}
Let $\{\eta(g)\}_{g=2}^{\infty}$ be any sequence of integers with $\eta(g)\in [1,2g-2]$, then
\[\liminf_{g\to \infty}  \sup_{X_g \in \sM_g} \spg_{\eta(g)}(X_g)\geq \frac{1}{4}.\]
\et

\bp
We will show that for any $\eta(g)$ with sufficiently large $g$, one can find a genus $g$ surface $X_{g}$ with  $\spg_{\eta(g)}(X_g)$ close to $1/4$. To see this, we split the proof into the following four cases.
\vspace{.1in}

\emph{Case-$1$: $\eta(g)=2g-2$.} 

Let $X_g(t): (0, 1)\to \sM_g$ be a family of closed hyperbolic surfaces such that $$\lim \limits_{t\to 0}X_g(t)=\underbrace{X_{0,3}\sqcup\cdots\sqcup X_{0,3}}_{\text{$2g-2$ copies}} \in \partial \sM_g.$$ First by \cite[Theorem 2]{OR09}, $\lambda_{2g-2}(X_g(t))\geq \frac{1}{4}$  for all $t\in (0,1)$. Secondly by Theorem~\ref{SWY-80} we know that $\lambda_{2g-3}(X_g(t))\to 0$ as $t\to 0$. Thus, 
\[\sup_{X_g \in \sM_g} \spg_{2g-2}(X_g)\geq \liminf \limits_{t\to 0}\spg_{2g-2}(X_g(t))\geq \frac{1}{4}.\]\vspace{.05in}

\emph{Case-$2$: $\eta(g)\in [g+1,2g-3]$.} 

First we choose a hyperbolic surface $Z_{1,2}\in \sM_{1,2}$ such that $\bar \lambda_1(Z_{1,2})\geq \frac{1}{4}$ by Theorem~\ref{t:0312}. Recall also that $\lambda_{1}(X_{0,3})\geq \frac{1}{4}$. By Lemma \ref{d-f-g-2} we can construct $X_g(t): (0,1) \to \sM_g$ as a family of closed hyperbolic surfaces such that $$\lim \limits_{t\to 0}X_g(t)=\underbrace{X_{0,3}\sqcup\cdots\sqcup X_{0,3}}_{\text{$i$ copies}} \sqcup\underbrace{Z_{1,2}\sqcup\cdots\sqcup Z_{1,2}}_{\text{$j$ copies}} \in \partial \sM_g$$ where $i$ and $j$ are two non-negative integers satisfying $i+j=\eta(g)$. By Theorem~\ref{SWY-80} we know that $\lim \limits_{t\to 0}\lambda_{\eta(g)-1}(X_g(t))= 0$.  By Proposition \ref{mmp-c} we have
\[\lim_{t\to 0}\lambda_{\eta(g)}(X_g(t))= \frac{1}{4}\]
which implies $$\sup \limits_{X_g \in \sM_g} \spg_{\eta(g)}(X_g)\geq \lim \limits_{t\to 0}\spg_{\eta(g)}(X_g(t))=\frac{1}{4}.$$\vspace{.05in}

\emph{Case-$3$: $\eta(g)\in [2,g]$.} 

Same as Case-$2$ we first choose a hyperbolic surface $Z_{1,2}\in \sM_{1,2}$ such that $\bar \lambda_1(Z_{1,2})\geq \frac{1}{4}$. Let $g_1>0$ be the integer determined in Lemma \ref{d-f-g-1}. Note that $g_{1}$ tends to $\infty$ as $g\to \infty$ because $2g_1\geq g-2$. Then by Theorem \ref{hm-o} we know that for any $\epsilon>0$ and large enough $g>0$, one may choose a hyperbolic surface $\mathcal{X}_{g_1,2}\in \sM_{g_1,2}$ such that 
\[\bar\lambda_1(\mathcal{X}_{g_1,2})>\frac{1}{4}-\epsilon.\] 

Fix any such large $g$, then by Lemma \ref{d-f-g-1} we construct  $X_g(t): (0,1)\to \sM_g$ as a family of closed hyperbolic surfaces such that
$$\lim \limits_{t\to 0}X_g(t)=\underbrace{X_{0,3}\sqcup\cdots\sqcup X_{0,3}}_{\text{$i$ copies}} \sqcup \underbrace{Z_{1,2}\sqcup\cdots\times Z_{1,2}}_{\text{$j$ copies}}\sqcup \mathcal{X}_{g_1,2}  \in \partial \sM_g$$
where $i$ and $j$ are two non-negative integers satisfying $i+j=\eta(g)-1$. By Theorem~\ref{SWY-80} we know that $\lim \limits_{t\to 0}\lambda_{\eta(g)-1}(X_g(t))= 0$. Applying the Min-Max Principle in Proposition \ref{mmp} to this sequence with $k=\eta(g)$ (note that $g$ is a fixed large integer hence $\eta(g)$ is also fixed), 
we have
\[\liminf_{t\to 0}\lambda_{\eta(g)}(X_g(t))\geq \min\{\bar \lambda_1(\sM_{0,3}), \bar \lambda_1(Z_{1,2}), \bar \lambda_1(\mathcal{X}_{g_1,2})\}\geq \frac{1}{4}-\epsilon\]
which implies 
\[\liminf_{g\to \infty}  \sup_{X_g \in \sM_g} \spg_{\eta(g)}(X_g)\geq\frac{1}{4}-\epsilon\]
because $\sup \limits_{X_g \in \sM_g} \spg_{\eta(g)}(X_g)\geq \liminf \limits_{t\to 0}\spg_{\eta(g)}(X_g(t))$. Since $\epsilon>0$ can be arbitrarily small, we have
\[\liminf_{g\to \infty}  \sup_{X_g \in \sM_g} \spg_{\eta(g)}(X_g)\geq\frac{1}{4}.\]
\vspace{.05in}

\emph{Case-$4$: $\eta(g)=1$.} 

This is due to Hide--Magee \cite[Corollary 1.3]{HM21} because $\spg_{1}(X_g)=\lambda_1(X_g)$.
\vspace{.1in}

The four cases above cover all possible $\eta(g)$ and hence complete the proof.
\ep

\begin{rem*}
The method in this paper works for  indices in the range of $[1,2g-2]$ in Theorem~\ref{mt-1}. The restriction comes from the lack of suitable components with $\lambda_{1}$ close to $1/4$ when we constructing the degenerating family. It would be interesting to know that whether the assumption $\eta(g)\in [1,2g-2]$ can be dropped. 
\end{rem*}

We also note that, together with \cite[Corollary 2.3]{Cheng-75}, the proof of Theorem \ref{mt-1} above actually gives that
\bt
For any $0\leq j<i$ where $i=o(\ln(g))$, 
\[\lim_{g\to \infty}  \sup_{X_g \in \sM_g} \left(\lambda_i(X_g)-\lambda_j(X_g)\right))= \frac{1}{4}.\]
\et

\bibliographystyle{amsalpha}
\bibliography{ref}

\end{document}